\documentclass[11pt]{article}%
\usepackage{eurosym}
\usepackage{bbm}
\usepackage{amsmath}
\usepackage{amsfonts}
\usepackage{amssymb}
\usepackage{graphicx}
\usepackage{pstricks}
\usepackage{pst-node}
\usepackage{pst-plot}
\usepackage[final]{showkeys}
\usepackage{comment}
\usepackage{hyperref}
\usepackage[a4paper]{geometry}
\newcommand{\C}{{\mathbb C}}

\newcommand{\R}{{\mathbb R}}
\newcommand{\Z}{{\mathbb {Z}}}

\newcommand{\T}{{\mathbb T}}

\newcommand{\Po}{{\mathcal P}}

\newcommand{\ba}{\begin{eqnarray}}
\newcommand{\ea}{\end{eqnarray}}
\newcommand{\bas}{\begin{eqnarray*}}
\newcommand{\eas}{\end{eqnarray*}}
\newcommand{\be}{\begin{equation}}
\newcommand{\ee}{\end{equation}}

\newcommand{\ii}{{\scriptstyle\sqrt{-1}}}

\newtheorem{theorem}{Theorem}
\newtheorem{proposition}[theorem]{Proposition}
\newtheorem{corollary}[theorem]{Corollary}
\newtheorem{lemma}[theorem]{Lemma}

\newenvironment{proof}[1][Proof]{\noindent\textbf{#1.} }{\ \rule{0.5em}{0.5em}}
\newtheorem{preacknowledgments}{Acknowledgments}

\newtheorem{preexample}[theorem]{Example}
\newenvironment{example}{\begin{preexample}\rm}{\hfill$\Diamond$\end{preexample}}
\newtheorem{preremark}[theorem]{Remark}
\newenvironment{remark}{\begin{preremark}\rm}{\hfill$\Diamond$\end{preremark}}
\newtheorem{preremarks}[theorem]{Remarks}
\newenvironment{remarks}{\begin{preremarks}\rm$\quad$}{\hfill$\Diamond$\end{preremarks}}
\newtheorem{prenotation}[theorem]{Notation}

\numberwithin{equation}{section}
\numberwithin{theorem}{section}
\begin{document}

\title{{Complex symplectomorphisms and pseudo-K\"ahler islands in the quantization of toric manifolds}}
\author{William D. Kirwin\thanks{Mathematics Institute, University of Cologne,
Weyertal 86 - 90, 50931 Cologne, Germany.\newline email:
will.kirwin@gmail.com}, Jos\'e M. Mour\~ao\thanks{Center
for Mathematical Analysis, Geometry and Dynamical Systems and the Department
of Mathematics, Instituto Superior T\'ecnico, Universidade de Lisboa, Av. Rovisco Pais, 1049-001
Lisbon, Portugal.\newline Lehrstuhl f\"ur Theoretische Physik III, Friedrichs-Alexander-Universit\"at, Erlangen-N\"urnberg\newline email: jmourao@math.tecnico.ulisboa.pt}, and Jo\~ao P. Nunes \thanks{Center
for Mathematical Analysis, Geometry and Dynamical Systems and the Department
of Mathematics, Instituto Superior T\'ecnico, Universidade de Lisboa, Av. Rovisco Pais, 1049-001
Lisbon, Portugal.\newline email: jpnunes@math.tecnico.ulisboa.pt}}
\maketitle

\date

\begin{abstract}
Let $P$ be a Delzant polytope. We show that the quantization of the corresponding toric manifold $X_{P}$ in toric K\"ahler polarizations and in the toric real polarization are related by analytic continuation of Hamiltonian flows evaluated at time $t = \ii s$. We relate the quantization of $X_{P}$ in two different toric K\"ahler polarizations by taking the time-$\ii s$ Hamiltonian ``flow'' of strongly convex functions on the moment polytope $P$. By taking $s$ to infinity, we obtain the quantization of $X_{P}$ in the (singular) real toric polarization.

Recall that $X_{P}$ has an open dense subset which is biholomorphic to $({\mathbb{C}}^{*})^{n}$. The quantization of $X_{P}$ in a toric K\"ahler polarization can also be described by applying the complexified Hamiltonian flow of the Abreu--Guillemin symplectic potential $g$, at time $t=\ii$, to an appropriate finite-dimensional subspace of quantum states in the quantization of $T^{*}{\mathbb{T}}^{n}$ in the vertical polarization. By taking other imaginary times, $t= k \ii, k\in {\mathbb{R}}$, we describe toric K\"ahler metrics with cone singularities along the toric divisors in $X_{P}$.

For convex Hamiltonian functions and sufficiently negative imaginary part of the complex time, we obtain degenerate K\"ahler structures which are negative definite in some regions of $X_{P}$. We show that the pointwise and $L^2$-norms of quantum states are asymptotically vanishing on negative-definite regions.

\end{abstract}
\tableofcontents

\section{Introduction}
\label{s1}

Let $(M,\omega)$ be a symplectic manifold such that $[\omega/2\pi]\in H^{2}(M,{\mathbb{Z}})$. The quantization of the classical phase space $(M,\omega)$ should produce an appropriate Hilbert space of quantum states. This is achieved in the framework of geometric quantization (see, for example, \cite{Woodhouse}) by considering prequantum data $(L,\nabla,h)$ which consists of a smooth complex line bundle $L\rightarrow X$ with hermitian structure $h$ and compatible connection $\nabla$ with curvature $-\ii \omega$. In addition to this data, one needs to specify a so-called polarization of $(M,\omega)$. This is an involutive Lagrangian distribution ${\Po}$ in the complexified tangent bundle of $M.$ The corresponding Hilbert space of quantum states ${\mathcal{ H}}_{\Po}$ consists of states $s$ satisfying $\nabla_Xs=0$ for all $X\in\Gamma(\Po).$

A central problem in geometric quantization is the study of the dependence of ${\mathcal{H}}_{\Po}$ on the choice of ${\Po}$. In particular, one looks for natural unitary identifications between the Hilbert spaces of quantum states associated to different polarizations.

Let $P$ be a Delzant polytope. The quantization of the corresponding compact symplectic toric manifold $X_{P}$ was studied in \cite{BFMN, KMN1}. There, it was shown that the quantization in the singular toric real polarization ${\Po}_{\mathbb{R}}$ associated to the fibers of the moment map
\[
\mu:X_{P}\rightarrow P
\]
can be continuously related to the quantization of $X_{P}$ in a toric invariant K\"{a}hler polarization determined by a toric invariant complex structure $I$. The K\"{a}hler structure $(X_{P},\omega,I)$ combines with the connection $\nabla$ to define an holomorphic structure on the prequantum bundle $L$, so that
\[
{\mathcal{H}}_{{\Po}_{I}}=H^{0}(X_{P},L),
\]
is the space of $I$-holomorphic sections of $L$. This space has a preferred basis of monomial sections, $\{\sigma^{m}:m\in P\cap\mathbb{Z}^{n}\}$, labeled by the integral points of $P$ (see, for example, \cite{Cox}). As shown in \cite{BFMN}, as the K\"{a}hler structure degenerates in an appropriate way, the K\"{a}hler polarizations converge to ${\Po}_{\mathbb{R}}$ while the $L^{1}$-normalized monomial holomorphic section $\sigma^{m},m\in P\cap{\mathbb{Z}}^{n}$, converges to a distributional section supported on $\mu^{-1}(m)$. These distributional sections generate ${\mathcal{H}}_{{\Po}_{\mathbb{R}}}$. In \cite{KMN1}, this continuous family of degenerating K\"{a}hler structures was further studied by the inclusion of the half-form correction, and it was shown that in this case the convergence to distributional sections is obtained by taking $L^{2}$-normalized holomorphic monomial sections.

An interesting way of changing polarization on a symplectic manifold is by pushforward by the flow of an Hamiltonian vector field. One then has explicit formulas for the time evolution along the flow of the polarized sections of the prequantum bundle $L$. (See, for example, the discussion of the Schr\"{o}dinger equation in \cite{K_qgf,SniatyckiGQ,Woodhouse}.) Recently, this idea has been extended to imaginary and complex time by considering the analytic continuation in time of the Hamiltonian flow. This seems to have a wide application in geometric quantization and sheds light on the relation between the quantization of symplectic manifolds in real and K\"{a}hler polarizations. Hamiltonian flows in complex time were studied in the context of quantum gravity by Thiemann \cite{Th1}. In \cite{HK,HK2}, the evolution of polarizations under complexified flows on cotangent bundles was studied.

Let $K$ be a compact Lie group. In \cite{KMN2,KMN3}, the quantization of the cotangent bundle $T^{\ast}K\cong K_{\mathbb{C}}$ with $K\times K$-invariant K\"{a}hler polarizations was studied by continuously relating it to the quantization in the vertical polarization via complexified Hamiltonian flows of convex $Ad$-invariant functions on $Lie(K)$. This gives a new perspective on the relation between the coherent state transform for Lie groups of Hall \cite{Ha1} and geometric quantization, as explored in \cite{Ha4,KW1,FMMN1,FMMN2,LS2,Lempert-Szoke10}. The case of abelian groups, with $K={\mathbb{T}}^{n},K_{\mathbb{C}}=({\mathbb{C}}^{\ast})^{n}$, is of particular relevance for the present paper's study of toric varieties, since these carry an holomorphic action of $({\mathbb{C}}^{\ast})^{n}$ with an open dense orbit.

Complexified Hamiltonian flows also play a fundamental role in recent studies on the geometry of the space of K\"ahler metrics and on the study of the relation between algebro-geometric notions of stability and the existence of constant scalar curvature metrics \cite{Se,CS,Do}. In fact, the rays of toric K\"ahler structures studied in \cite{BFMN, KMN1}, and which in the present paper are studied from the perspective of Hamiltonian evolution in complex time, are geodesic rays in the space of toric K\"ahler metrics. (See, for example, \cite{Rubinstein-Zelditch1,Rubinstein-Zelditch2,SZ}.)

In this paper, we apply Thiemann's complexifier method to the quantization of toric manifolds. Further applications in other contexts will appear in \cite{MN2}.

By taking a strongly convex function on $P$ as complexifier we relate the K\"ahler quantizations of $X_{P}$ along a geodesic ray of K\"ahler structures. By taking imaginary time to $+\ii\infty$, these K\"ahler quantizations are then related to the quantization of $X_{P}$ associated to ${\Po}_{\mathbb{R}}$. In this way, we rederive the results of \cite{BFMN, KMN1} through Hamiltonian evolution in complex time.

By taking this Hamiltonian flow to sufficiently negative imaginary time, one can create regions where the metric on $X_{P}$ is no longer positive definite. The polarized sections of $L$ then develop interesting behaviour since the derivatives of pointwise norms are very sensitive to the positivity of the metric.

Recall that $X_{P}$ has an holomorphic action of $({\mathbb{C}}^{*})^{n}$ with an open dense orbit $\check X_{P}$. There is an analog of the continuous family of K\"ahler polarizations of a cotangent bundle of a compact Lie group, starting at the vertical polarization, which was considered in \cite{FMMN1,FMMN2,KW2}. We use Thiemann's complexifier approach to define an isomorphism of a finite-dimensional subspace of the quantization of $T^{*}{\mathbb{T}}^{n}$ with respect to the vertical polarization with the K\"ahler quantization of $X_{P}$. In this case, the role of the complexifier is played by an Abreu--Guillemin symplectic potential $g$, which is smooth on $\check X_P$ but only continuous along $X_{P}$. By taking the complexified Hamiltonian flow for this singular Hamiltonian to imaginary time $k\ii, k\in{\mathbb{R}}$, we obtain singular K\"ahler structures on $X_{P}$ whose K\"ahler metrics have cone angles at the toric divisors in $X_{P}$.

\section{Preliminaries\label{s2} on toric manifolds}
\label{sec:prelim}

In this section, we review some facts concerning toric varieties. For more details see e.g. \cite{CdS,Abreu03,BFMN,KMN1}. Let $(X_{P},\omega)$ be a smooth compact symplectic toric manifold with symplectic form $\omega$ such that $\left[\frac{\omega}{2\pi}\right]-\frac{c_{1}(X_{P})}{2}\in H^{2}(X,{\mathbb{Z}})$. Denote the moment map by
\[
\mu:X\rightarrow Lie(\mathbb{T}^{n})^{\ast}\simeq\mathbb{R}^{n}
\]
and let $P=\mu(X_{P})$ be the (Delzant) moment polytope with associated fan $\Sigma$.

Let $\check{P}$ denote the interior of the moment polytope $P$. Consider action-angle coordinates $(x,\theta)$ on $\check{X}_P\cong\check{P}\times{\mathbb{T}}^{n}\cong({\mathbb{C}}^{*})^{n}$, so that $\mu({x},\theta) = x = \,(x^{1},\dots,x^{n})$. The symplectic form $\omega$ in this coordinate chart is simply
\begin{equation}
\left.  \omega\right\vert _{\mu^{-1}(\check{P})}=\sum_{j=1}^{n} dx^{j}\wedge d\theta_{j}. \label{eqn:interior symp}
\end{equation}
The condition $[\frac{\omega}{2\pi}]-\frac{c_{1}(X)}{2}\in H^{2}(X_{P},{\mathbb{Z}})$ allows us to choose the moment polytope $P= \mu(X)$ such that
\begin{equation}
\label{t11}
P = \left\{  x \in{\mathbb{R}}^{n} \ : \ \ell_{j}(x) = \nu_{j} \cdot x + \lambda^{j} \geq0, \ j = 1,\dots,r\right\},
\end{equation}
with
\begin{equation}
\label{halflambdas}
\lambda_{j} \in\frac12 + {\mathbb{Z}}, \,\,\, j=1\dots,r,
\end{equation}
and where $\nu_{j}$ is the primitive inward pointing normal vector to the $j$-th facet.

The formalism of Guillemin and Abreu describes toric K\"ahler structures in symplectic coordinates. A torus-invariant complex structure on $(X_{P},\omega)$ is determined by a symplectic potential $g=g_{P}+\varphi$ (also called an Abreu--Guillemin potential), where $g_{P}\in C^{\infty}(\check{P})$ is
\begin{equation}
g_{P}({x})=\frac{1}{2}\sum_{j=1}^{r}\ell_{j}({x})\log\ell_{j}({x}),
\label{eqn:gP}
\end{equation}
and $\varphi$ is a smooth function on $P$ such that the Hessian $H_{g_{P}+\varphi}$ is positive definite on $\check{P}$ and satisfies the regularity conditions
\begin{equation}
\det{H}_{g_{P}+\varphi}=\left(  \alpha({x})\prod_{j=1}^{r}\ell_{j}({x})\right)  ^{-1}, \label{eqn:greg}
\end{equation}
where $\alpha$ is smooth and strictly positive on $P$ \cite{Guillemin94,Abreu03}. Denote the set of such functions by $C_{+}^{\infty}(P)$ .

In the symplectic frame determined by the action-angle coordinates $({x,\theta})$ on $\check{X}_P$, the complex structure $I_{g}$ and the metric $\gamma_{g}=\omega(\cdot,I_{g}\cdot)$ associated to the symplectic potential $g$ are
\begin{equation}
\label{eqn:KahlerStruct}
I_{g}=
\begin{pmatrix}
0 & -H_{g}^{-1}\\
H_{g} & 0
\end{pmatrix}
\text{ and }\gamma_{g}=
\begin{pmatrix}
H_{g} & 0\\
0 & H_{g}^{-1}
\end{pmatrix}
.
\end{equation}

The symplectic potential $g$ fixes a biholomorphism $(\check{X}_{P},I_{g})\cong\check{P}\times\mathbb{T}^{n}\cong(\mathbb{C}^{\ast})^{n}$ given by
\begin{equation}
\label{eqn:w}
({x,\theta)}\in\check{P}\times\mathbb{T}^{n}\mapsto w=(w_{1},\dots,w_{n}) =e^{y+i{\theta}}=(e^{y_{1}+i\theta_{1}},\dots,e^{y_{n}+i\theta_{n}})\in(\mathbb{C}^{\ast})^{n},
\end{equation}
where $y_{j}=\partial g/\partial x^{j}$. The map ${x}\mapsto{y}=\partial g/\partial{x}$ \ is a bijective Legendre transform. The inverse map is given by ${x}=\partial h/\partial{y}$, where $h$ is a K\"{a}hler potential given in terms of $g$ by
\[
h:={x}\cdot{y}-g.
\]
This biholomorphism extends uniquely to a biholomorphism of $X_{P}$ with a complex toric variety $Y$ with fan $\Sigma$ associated to $P$. Henceforth, we will assume that $X_{P}$ is equipped with a K\"{a}hler structure $(X_{P},\omega,I_{g},\gamma_{g})$ determined by $\omega$ and by a symplectic
potential $g$.

The Picard group of $X_{P}$ is generated by linear equivalence classes of the torus invariant irreducible divisors $D_{1},\dots,D_{r}$, where $D_{j}, j=1,\dots, r$, is the toric divisor corresponding to the preimage under the moment map $\mu$ of the $j$-th facet of $P$.

Consider a divisor $D^{L}=\lambda_{1}^{L} D_{1}+\dots+\lambda_{r}^{L} D_{r}$, for $\lambda_{1}^{L},\dots,\lambda_{r}^{L} \in{\mathbb{Z}}$, defining a holomorphic line bundle $L=\mathcal{O}(D^{L})$ on $X_{P}$. Let $\sigma_{D^{L}}$ be the (unique up to constant) meromorphic section of $L$ with divisor $D^{L}$.

From \cite{Cox}, the divisor of the (rational) function defined on the open orbit by
\[
w^{m}=w_{1}^{m_{1}}\cdots w_{n}^{m_{n}},\,\, m=(m_{1},\dots,m_{n})\in{\mathbb{Z}}^{n},
\]
can be computed to be
\begin{equation}
\label{eqn:Coxformula}
\operatorname{div}(w^{m}) =\sum_{j=1}^{r} \left\langle \nu_{j},m\right\rangle
D_{j}.
\end{equation}
One then has
\begin{align}
\label{hzero}
H^{0}(X_P,L)  &  = \mathrm{span}_{\mathbb{C}} \left\{  w^{m} \sigma_{D^{L}}: m\in{\mathbb{Z}}^{n},\,\operatorname{div}(w^{m} \sigma_{D^{L}}) \geq0 \right\}\nonumber\\
&  = \mathrm{span}_{\mathbb{C}} \left\{  w^{m} \sigma_{D^{L}}: m\in
{\mathbb{Z}}^{n}, \langle m,\nu_{i}\rangle+\lambda_{i}^{L}\geq0, i=1,\dots,r
\right\}.
\end{align}
Therefore, there is a natural bijection between the torus invariant basis of $H^{0}(X_{P},L)$ and the integral points of the Delzant polytope $P_{L}$,
\begin{equation}
\label{eqn:p_Ldef}
P_{L}:=\{x\in{\mathbb{R}}^{n}:\left\langle x,\nu_{j}\right\rangle +\lambda_{j}^{L} \geq0,~j=1,\dots,r\}\subset{\mathbb{R}}^{n}.
\end{equation}

For simplicity, let us assume that $L$ is ample so that there is a canonical bijection between the vertices of $P_{L}$ and the vertices of $P$, defined by the equality of the set of normals of the facets meeting at those vertices.

Let $L\to X_{P}$ be the smooth line bundle with first Chern class
\[
c_{1}(L)= \left[\frac{\omega}{2\pi}\right]-\frac{c_{1}(X_{P})}{2},
\]
equipped with a  prequantum connection $\nabla$ with curvature $-\ii\left(\omega-\tfrac12\rho\right)$, where $\rho$ is the Ricci form. $L$ is the ``half-form'' corrected prequantum line bundle, as in \cite{KMN1}. Recall that $[-\frac{\rho}{2\pi}]$ is the Chern class of the canonical bundle $K_{X_{P}}$ of $X_{P}$ for a compatible complex structure; if $c_{1}(X_{P})$ is even, then $L$ is the tensor product of the prequantum line bundle of curvature $-\ii\omega$ with a square root of $K_{X_{P}}$.

With the above choice of moment polytope, there are no integral points on the boundary of $P$ and to each integral point inside $P$ there corresponds a half-form corrected K\"ahler polarized state, \textit{i.e.} a holomorphic section in $H^{0}(X_{P},L)$. Such sections are represented as products \begin{equation}
\label{secs}
s^{m} = \sigma^{m} \otimes\sqrt{dZ}, \,\, m\in P\cap{\mathbb{Z}}^{n}.
\end{equation}
of multivalued sections $\sigma^{m} $ corresponding to the integral points of $P$ with a canonical multivalued section $\sqrt{dZ}$ of the square root of the canonical bundle $\sqrt{K_{X_{P}}}$ with divisor
\[
div(\sqrt{dZ}) = - \frac12 \left(  D_{1} + \dots+ D_{r}\right).
\]

\begin{remark}
In \cite{KMN1}, the inclusion of half-forms was treated in a different, but equivalent, way involving a separation of the ``norm'' and ``phase'' parts of the (would-be) half-form bundle. This treatment was necessary because the square root of the canonical bundle of a toric manifold may not always exist.  It is straightforward to check that this formalism yields the same results as we obtain here, where by writing the half-form corrected monomial sections as $w^{m}e^{-h}{\mathbbm 1}^{U(1)}\otimes\sqrt{dZ}$ one is implicitly allowing each of the factors in the tensor product to be a ramified section. In any case, the tensor product of the two factors is well defined and is a smooth section of a line bundle with Chern class $\left[\frac{\omega}{2\pi}\right]-\frac{c_{1}(X_{P})}{2}$.
\end{remark}

\section{Complex-time evolution in geometry}
\label{sec:cpxTgeom}

In this section, we briefly review some aspects of the general theory of complex-time Hamiltonian evolution in geometry. We first describe complex-time evolution in cotangent bundles and then in more general symplectic manifolds.

\subsection{The Thiemann complexifier method}

The idea, which we describe now, is due to the work of Thiemann \cite{Th1} and has come to be known as the Thiemann complexifier method. Consider a cotangent bundle $T^{\ast}M$ equipped with its canonical symplectic structure. Fix a function $h\in C^{\infty}(T^{\ast }M)$, called the \emph{complexifier function}. Let $X_h$ be its Hamiltonian vector field. A complex structure is defined on $T^{\ast}M$ by declaring that given a function $f:M\rightarrow\mathbb{C}$, the function
\begin{equation}
\label{eqn:Thiemann1}
f_{\mathbb{C}}:=e^{\ii X_h}f
\end{equation}
is \emph{holomorphic.} Here, $e^{\ii X_h}$ can be understood as a power series $\sum_{j=1}^{\infty}\frac{\ii^{j}}{j!}X_{h}^{j}f$. One may interpret (\ref{eqn:Thiemann1}) as saying that holomorphic functions are defined to be the composition of vertically constant functions with the time-$\ii$ Hamiltonian ``flow'' of $h$ (which is, of course, not a flow in the usual sense).
Since $X_h$ is a derivation of the Poisson bracket, we have that for any two functions $f_{1},f_{2}:M\rightarrow\mathbb{C}$,
\[
\{e^{\ii X_h}f_{1},e^{\ii X_h}f_{2}\}=0,
\]
whence the symplectic form is type $(1,1)$ with respect to the complex structure.

In general there is no reason that (\ref{eqn:Thiemann1}) should be a convergent power series, and indeed in general it is not. Nevertheless, the Thiemann complexifier method can be made rigorous in many interesting situations.

The first fundamental observation is that the geodesic flow on a compact real-analytic Riemannian manifold $M$ can be analytically continued and the resulting geometry is the adapted complex structure due to Lempert and Sz\H{o}ke, and to   Guillemin and Stenzel \cite{LS, Szoke-91, GS1, GS2}. There are many ways to describe a complex structure on a symplectic manifold---in terms of its $(1,0)$-tangent bundle, the almost complex tensor $J$, holomorphic functions, or holomorphic sections, for example---and for each of these descriptions, there is a version of the following theorem (due to the first author and Hall \cite{HK}) which explains precisely what it means to analytically continue the geodesic flow. We state here a version in terms of the $(1,0)$-bundle, i.e., in terms of the corresponding holomorphic polarization of $T^*M$, as that is how we will primarily describe toric invariant complex structures. Since the geodesic flow is the Hamiltonian flow of half the norm-squared function, the content of the following theorem, at least in the present context, is that the Thiemann complexifier method ``works'' if one takes $E(x,p):=\frac{1}{2}\left\vert p\right\vert ^{2}$ as a complexifier function.

Denote the time-$\sigma$ geodesic flow by $\Phi_{\sigma}:T^{\ast}M\rightarrow T^{\ast}M$ and the vertical tangent bundle of $T^*M$ by $\operatorname{Vert}(T^*M).$

\begin{theorem}\
\label{thm:HKmain}
\cite[Theorem 1]{HK}
For every $\varepsilon>0$ there exists $R>0$ such that for each $(x,p)\in T^{\ast}M$ with $\left\vert p\right\vert <R,$ the family of Lagrangian subspaces
\[
\sigma\in\mathbb{R\mapsto} \mathcal{P}_{\sigma}(x,p):=\left(  \Phi_{-\sigma}\right)_{\ast}\operatorname{Vert}_{(x,p)}^{\mathbb{C}}(T^{\ast}M) \subset T_{(x,p)}^{\mathbb{C} }T^{\ast}M
\]
can be analytically continued to a disk $D_{1+\epsilon}$ of radius $1+\varepsilon$ about the origin in $\mathbb{C}$. The resulting complex Langrangian distribution $\mathcal{P}_\ii$ on $T^{\ast,R}M:=\{(x,p):\left\vert p\right\vert <R\}$ is integrable, satisfies $P(\ii)\cap\overline{P(\ii)}=\{0\}$, and is negative with respect to the canonical symplectic form; that is, $\mathcal{P}_\ii$ is the $(0,1)$-tangent bundle of a K\"{a}hler complex structure on $T^{\ast,R}M$.
\end{theorem}

An equivalent theorem is that each real-analytic metric on $M$ provides, via Thiemann complexification with respect to the associated norm-squared function, a biholomorphism of a tubular neighborhood of $M$ in $T^{\ast}M$ with a tube in the Bruhat--Whitney complexification of $M$.

\bigskip
In \cite{KMN2, KMN3}, the authors studied Thiemann's complexifier method in the case of $T^{\ast}K$ where $K$ is a compact Lie group. The main geometric result is that there is an infinite-dimensional family $\mathfrak{H}(K)$ of Hamiltonians on $T^{\ast}K$ (which includes $E$) for which Thiemann complexification makes sense. Here, we can state the result precisely in terms of the complexification $K_{\mathbb{C}}$: each of the Hamiltonians in the family defines a diffeomorphism $T^{\ast}K\simeq K_{\mathbb{C}}$ such that the pullback of the standard complex structure on $K_{\mathbb{C}}$ is compatible with the canonical symplectic structure on $T^{\ast}K$.

One final remark before we return to the case at hand---toric manifolds---is that there is nothing really special about time $\ii$. In all of the above cases, the time-$\ii t$ Thiemann complexification works just as well as long as $t>0$. Moreover, the limit as $t\rightarrow0$ of the time-$\ii t$ complex structure yields the vertical polarization; it is a kind of degeneration of the complex structure. Lifting this to the quantum level yields what we call ``decomplexification''.

As $t$ becomes negative, though, the resulting complex structures are generally negative with respect to the symplectic form, and quantization with respect to a negative complex structure is generally not well behaved. However, one of the main results of this article is to show that for a large family of complexifiers on toric manifolds (the Abreu--Guillemin potentials), one can in fact consider $t<0$ and, although the resulting geometry include negative K\"ahler regions, we will see below in Section \ref{sssik} the associated quantization is nevertheless well behaved.

\subsection{General complex-time evolution}

Complexified Hamiltonian flows in more general symplectic manifolds have been used both in quantization and geometry. In the context of K\"ahler geometry, they describe geodesics on the space of K\"ahler metrics in a fixed cohomology class $[\omega/2\pi]\in H^2(M)$, with respect to the Donaldson--Mabuchi--Semmes metric. Hence, they play an important role in the question of existence and uniqueness of K\"ahler metrics with special properties, such as constant scalar curvature metrics or K\"ahler--Einstein metrics, and are related with properties of algebro-geometric stability. (For reviews, see \cite{Do,PS}.)  In the context of geometric quantization, as described above and below, complexified Hamiltonian flows can be used to relate quantizations of a given symplectic manifold $(M,\omega)$ with respect to different polarizations. For example, in some cases, one can relate quantizations in a real polarization with quantizations along families of K\"ahler polarizations.

In \cite{MN1}, the following approach to complex-time flows was taken. On $(M,J_0)$, one acts with the complex-time $\tau$ flow $e^{\tau X_h}$ (of a real-analytic vector field $X$) directly on systems of local $J_0$-holomorphic coordinates, thereby changing the sheaf of holomorphic functions and obtaining a new complex structure $(M,J_\tau)$ which is biholomorphic to the initial one. In general, for compact $M$, this procedure is well defined so long as $|\tau|$ is sufficiently small. In the case of a K\"ahler manifold $(M,\omega,J_0)$ and of the flow of a real-analytic Hamiltonian vector field $X_h$, this construction yields a local one-parameter family of inequivalent K\"ahler structures $(M,\omega,J_\tau)$.

\begin{theorem}\ \cite{MN1}
Let $(M,\omega,J_0)$ be a compact real-analytic K\"ahler manifold (here, $\omega$ and $J_0$ are real analytic). Let $h\in C^\omega(M)$. There exists $T>0$ such that for all $\tau \in \C,\,\, |\tau|<T$,
\begin{enumerate}
\item[a)] the action of the complex-time Hamiltonian flow of $X_h$ on the sheaf of $J_0$-holomorphic functions defines a K\"ahler structure $(M,\omega,J_\tau)$, and
\item[b)] the corresponding K\"ahler polarization is given by
\[
{\mathcal P}_\tau = e^{\tau {\mathcal L}_{X_h}} {\mathcal P}_0,
\]
where ${\mathcal P}_0$ is the K\"ahler polarization defined by $J_0$. The polarization $\Po$ can be interpreted as a convergent power series in $\tau$ if $e^{\tau {\mathcal L}_{X_h}}$ is applied to appropriate local sections of ${\mathcal P}_0$, such as $\{X_{z_i}\}_{i=1,\dots,n}$, where $\{{z_i}\}_{i=1,\dots,n}$ is a system of local $J_0$-holomorphic coordinates.
\end{enumerate}
\end{theorem}

Moreover, one can obtain an explicit formula for the time $\tau$ evolution of the K\"ahler potential on $M$. This approach to the description of complexified Hamiltonian flows is related to the one explored in \cite{BLU}, where the complexified flow is constructed first in a complexification $M_\C$ of $M$ and then projected to $M$.

Below, we will see how complex-time Hamiltonian evolution can be applied in the context of toric geometry and subsequently in the quantization of symplectic toric manifolds.

\subsection{Complex-time flow of toric geometries}
\label{sstfp}

Let $g$ be a symplectic potential for $X_{P}$ and let ${\Po}_{g} = T^{0,1}X_{P}$ be the corresponding K\"ahler polarization of $(X_{P},\omega)$, so that
\[
{\Po}_{g} =\mathrm{span}_{\mathbb{C}}\,\,\left\{  \frac{\partial
}{\partial \bar z_{j}}, j=1,\dots,n\right\}  ,
\]
where $z_{j}=y_{j}+i\theta_{j}, j=1,\dots,r$ with $y_j = \partial g/\partial x^j.$ Let ${\Po}_{\mathbb{R}}$ be the (singular) toric real polarization of $(X_{P},\omega)$ given at each point by the complexified tangent space to the fibers of the moment map $\mu$,
\[
{\Po}_{\mathbb{R}} = \mathrm{span}_{\mathbb{C}}\,\,\left\{
\frac{\partial}{\partial\theta_{j}}, j=1,\dots,n\right\}  .
\]

Let $\psi$ be a smooth strongly convex function on $P$ so that $s\psi\in C^{\infty}_{+}(P)$ for all $s>0$. In \cite{BFMN}, the family of toric K\"ahler structures determined by $g_{s} = g +s\psi$ was studied.

\begin{lemma}\ \cite{BFMN}
Pointwise, in the positive Lagrangian Grassmannian of $(\check{X}_{P},\omega)$, \[
lim_{s\to\infty} {\Po}_{s} = {\Po}_{\mathbb{R}}.
\]
\end{lemma}

Let $X_{\psi}= -\sum_{j=1}^{n} \frac{\partial\psi}{\partial x^{j}} \frac{\partial}{\partial\theta_{j}}$ be the Hamiltonian vector field for $\psi$. The following is a toric analog of Theorem 3.10 of \cite{KMN1}.

\begin{theorem}
\label{proppolspsi}
Let $s>0$.
\begin{enumerate}
\item[a)] As distributions, ${\Po}_{g_{s}} = e^{\ii s{\mathcal{L}}_{X_{\psi}}}{\Po}_{g}$.

\item[b)] If $dZ_{s}$ denotes the $I_{g_s}$-meromorphic section of $K_{X_{P}}$ with divisor $-(D_{1}+\cdots+D_{r})$, then $dZ_{s} = e^{\ii s {\mathcal{L}}_{X_{\psi}}} dZ,$ where the right-hand side makes sense (pointwise) as a power series in $s$.
\end{enumerate}
\end{theorem}

\begin{proof}
Since the symplectic form $\omega$ is of type $(1,1)$ with respect to the complex structure determined by $g_s$, we have that the $(0,1)$-part of the tangent space of $X_P$ relative to the complex structure $I_{g_s}$ is pointwise generated by the Hamiltonian vector fields associated to the corresponding  holomorphic coordinates $z_s^j, j=1,\dots,n$. These are given by
\[
z_{j}(s)= y_{j}+s\frac{\partial\psi}{\partial x^{j}}+ \ii\theta_{j}=y_{j}(s) + \ii\theta_{j},
\]
where $y_{j}(s)=\frac{\partial g_{s}}{\partial x^{j}}.$

For a smooth function $h\in C^\infty(X_P,\omega)$ let $X_h$ be the associated Hamiltonian vector field. We then have, for smooth functions $f\in C^\infty(X_P,\omega)$ and for real time $t$,
\[
X_{e^{tX_h}\cdot f} = e^{t{\mathcal L}_{X_h}}X_f,
\]
where $e^{tX_h}\cdot f$ and $e^{t{\mathcal L}_{x_h}}X_f$ represent, respectively, the pullback of $f$ by the flow of $X_h$ and the pushforward of $X_f$ by the flow of $-X_h$ at time $t$. In many interesting situations (such as the present case of toric varieties) these expressions are analytic in $t$ and admit analytic continuation to complex-time $\tau$. (See \cite{HK,HK2,KMN2,MN1}.)

In fact, $X_{\psi}(z_{j})= -\ii\frac{\partial\psi}{\partial x^{j}}$ and $X_{\psi}(X_{\psi}z_{j})=0$ so that
\[
e^{\ii s {\mathcal{L}}_{X_{\psi}}} z_{j} = y_{j} +s \frac{\partial\psi}{\partial x^{j}} + \ii\theta_{j} = z_{j}(s).
\]
It follows that ${\Po}_{g_{s}}$ is generated by $X_{z_j(s)} =  e^{\ii s{\mathcal L}_{X_{\psi}}}X_{z_j} = X_{e^{\ii sX_\psi}\cdot z_j}, \,\,j=1,\dots n,$ which proves $(a).$

To prove $(b)$, note that since the exterior derivative and pullback commute, and since $dZ= dz_{1}\wedge\cdots\wedge dz_{n}$, we obtain $e^{\ii s {\mathcal{L}}_{X_{\psi}}} dZ = dZ_{s}.$
\end{proof}

\subsection{Compactifying singular flows to toric geometries}

We will now consider the time $t=\ii$ Hamiltonian evolution for the Hamiltonian vector field $X_{g}$ associated to the symplectic potential $g$. This is a singular Hamiltonian function which is not differentiable on the boundary $\partial P$. However, its complexified Hamiltonian flow is still useful as it describes K\"ahler toric polarizations on $X_P$ starting from a real polarization.

The Hamiltonian vector field of $g$ on $\check X_P$ is given by
\begin{equation}\label{hamg}
X_{g} = -\sum_{j=1}^{n} y_{j} \frac{\partial}{\partial\theta_{j}}.
\end{equation}
On $\check X_{P}$, consider the vertical polarization, which is spanned by the vectors $\partial/\partial x^j,\ j=1,\dots,n.$

\begin{theorem}
\label{timeiflowprop}
\
\begin{enumerate}
\item[a)] Interpreted as an infinite series, the operator $e^{\ii {\mathcal{L}}_{X_{g}}}$ applied to the vector field $\frac{\partial}{\partial x^{j}}$ converges, and as distributions, ${\mathcal{P}}_{g} = e^{-\ii {\mathcal{L}}_{X_{g}}} {\mathcal{P}}_{v}.$

\item[b)] $dZ = \ii^{n} e^{\ii {\mathcal{L}}_{X_{g}}} d\theta_{1}\wedge\cdots\wedge d\theta_{n}.$
\end{enumerate}
\end{theorem}

\begin{proof}
First, $\left[  X_{g},\frac{\partial}{\partial x^{j}}\right]  = -\sum_{i=1}^{n}
(H_{g})_{ji} \frac{\partial}{\partial\theta_{i}}$ and $\left[  X_{g}, \left[
X_{g},\frac{\partial}{\partial x^{j}}\right]  \right]  =0,$ so that
\[
e^{-\ii{\mathcal{L}}_{X_{g}}} \frac{\partial}{\partial x^{j}} = \frac{\partial}{\partial x^{j}} + \ii \sum_{i=1}^{n} (H_{g})_{ji} \frac{\partial}{\partial\theta_{i}} = 2\sum_{i=1}^{n} (H_{g})_{ij} \frac{\partial}{\partial\bar z_{j}},
\]
which proves $(a)$.

Since, $e^{\ii {X_{g}}} \theta_{j} = -\ii z_j, j=1,\dots n,$ we have $e^{\ii {\mathcal{L}}_{X_{g}}} d\theta_{j} = d\theta_{j} - \ii dy_{j} = -i dz_{j},$
which implies $(b)$.
\end{proof}

\subsection{Cone-angle metrics on toric manifolds}

Let $(X,\omega,I)$ be a K\"ahler manifold and let $D\subset X$ be a divisor. A (singular) K\"ahler metric is said to have cone angle $2\pi \beta$ along $D$ if it is smooth along $X\setminus D$ and if at each point in $D$ there is a local coordinate chart $(U,(z_1,\dots,z_n))$ such that $D\cap U$ is given by $z_1=0$ and the metric locally is uniformly
equivalent to
\[
\ii\sum_{j=1}^n dz_j\wedge d\bar z_j + \ii |z_1|^{(2\beta-2)} dz_1\wedge d\bar z_1.
\]
(See, for example, \cite{CDS,DGSW})

Let $P$ be the Delzant polytope
\[
P=\{x\in \R^n:l_i(x)=\langle x, \nu_i\rangle - \lambda_i\geq 0, i=1,\dots ,r\}.
\]
As shown in  \cite[Prop. 2.1]{DGSW}, a K\"ahler metric on $X_P$ with cone angle singularity $2\pi\beta_i,\ 0<\beta_i\leq 1$, at the toric divisor $D_i$ corresponding to $l_i=0$ is described by the symplectic potential
\[
g_\beta = \tfrac12\sum_{j=1}^r \beta_j^{-1}l_j\log l_j+\psi
\]
for some $\psi\in C^\infty(P)$. This generalizes the construction of toric K\"ahler metrics on toric orbifolds \cite{LT,Ab2,BGL}.

In this section, we will see that K\"ahler toric cone angle metrics can be described by the complex-time Hamiltonian flow of the symplectic potential $g_\beta$ applied to the vertical polarization of the dense orbit $(\C^*)^n$.

\begin{theorem}\label{coneanglethm}
Let $((\C^*)^n,\omega,I_s)$ be the K\"ahler structure on $(\C^*)^n$ obtained from the vertical polarization of $\check P \times \mathbb{T}^n$
by the time $\ii s$ flow of $X_{g_\beta}$. Then, for $s>0$, $((\C^*)^n,\omega,I_s)$ compactifies to a cone angle metric on $X_P$ with cone angle $2\pi\beta/s$.
\end{theorem}

\begin{proof}
Consider the coordinates $\theta=(\theta_1,\dots,\theta_n)$ on $\T^n$. As in the proof of Proposition \ref{timeiflowprop}, we have
\[
e^{\ii s X_{g_\beta}}\cdot \theta_j = \theta_j -\ii s \frac{\partial g_\beta}{\partial x^j}.
\]
Therefore, the vertical polarization evolves under the complex-time Hamiltonian flow to the polarization defined by the complex structure given by the local holomorphic coordinates
\[
z_j(s)=  s\frac{\partial g_\beta}{\partial x^j} + \ii \theta_j, \,\, j=1,\dots n.
\]
In symplectic coordinates $(x,\theta)$ on the open dense orbit, the metric will have the familiar form
\[
\gamma_{s,\beta} = s \sum_{i,j=1}^n [H_{g_\beta}]_{ij} dx^i\otimes dx^j + \frac{1}{s} [H_{g_\beta}]_{ij}^{-1} d\theta_i\otimes d\theta_j.
\]
Since $sH_{g_\beta} = H_{g_{\beta/s}},$ the result follows.
\end{proof}

\bigskip
By applying the complex-time flows of singular Hamiltonians of the form $g_i=\frac12l_i\log l_i$ in succession, we obtain the following.

\begin{corollary}
Let $((\C^*)^n,\omega,I_\beta)$ be the K\"ahler structure on $(\C^*)^n$ obtained from the vertical polarization of $\check P\times\mathbb{T}^n$ by the chain of complex-time Hamiltonian flows given by
\[
e^{\frac{\ii}{\beta_i} X_{g_1}} \circ \cdots \circ e^{\frac{\ii}{\beta_d}X_{g_r}}
\]
where $g_i=\frac12 l_i\log l_i$. Then, $((\C^*)^n,\omega,I_\beta)$ compactifies to a cone angle toric metric on $X_{P}$.
\end{corollary}

\begin{proof}The proof follows exactly as in the proof of Theorem \ref{coneanglethm} and from the fact that the Hamiltonian flows of $X_{g_j}$ commute with each other for $j=1,\dots,r$.
\end{proof}

\section{Quantization of complex-time flows}
\label{strg}

In this section, we quantize the complex-time flows discussed in Section \ref{sec:cpxTgeom} and relate the results to the case of compact Lie groups. In particular, we lift the complex-time ``flows'' to the (half-form corrected) prequantum bundle.

\subsection{Prequantum operators and analytic continuation}
\label{sec:preQops}

Consider a symplectic manifold $(M,\omega)$ equipped with a prequantum bundle $L\rightarrow M$. The lift of the Hamiltonian vector field $X_h$ of $h\in C^\infty(M)$ to the prequantum bundle is the horizontal lift of $X_h$ plus $h\cdot \ii w\partial_w$, where $\ii w\partial_w$ is the canonical $\mathbb{C}^\times$-invariant Liouville vector field on the fibers of $L$.

The flow of this lifted vector field induces an action on sections of $L$ which is given infinitesimally by the Kostant--Souriau prequantum operator $\hat{h}:=\ii\nabla_{X_h} +h$. For $t\in\mathbb{R}$, denote the time-$t$ flow of the lifted vector field by $e^{-\ii t\hat{h}}:\Gamma(L)\rightarrow\Gamma(L)$. As in Section \ref{sec:cpxTgeom} (Theorems \ref{thm:HKmain} and \ref{proppolspsi}), we wish to analytically continue this flow to a map $D_{1+\epsilon}\times\Gamma(L)\rightarrow\Gamma(L)$ for some $\epsilon>0$.

One does not expect in general that the analytic continuation of this lifted flow yields a unitary map between sections of the prequantum bundle. For example, for $M=T^*K$ where $K$ is a compact Lie group $K$ equipped with a bi-invariant metric and $h$ equal to half of the norm-squared function on the fibers, the time-$\ii$ flow is analytic continuation with respect to the standard complex structure on $K_\mathbb{C}\simeq T^*K$, which is obviously not unitary.

On the other hand, let $\mathcal{H}_0$ be the Hilbert space for the half-form-corrected vertical quantization of $T^*K$ and $\mathcal{H}_\ii$ the Hilbert space for the K\"ahler quantization of $T^*K\simeq K_\mathbb{C}$. In \cite{Ha4}, Hall showed that his generalized Segal--Bargmann transform can be constructed as the unitary isomorphism between $\mathcal{H}_0$ and $\mathcal{H}_\ii$ which arises from the BKS pairing. The generalized Segal--Bargmann transform can be expressed in terms of the time-$\ii$ flow of half the norm-squared function, which we denote by $h$, as
\[
e^{\hat{h}}\circ E(\ii, h):\mathcal{H}_0\rightarrow\mathcal{H}_\ii,
\]
where $E(\ii,h):=e^{\tfrac12\left(\Delta-|\rho|^2\right)}$ (here, $\rho$ is half the sum of the positive roots of $K$) \cite{KMN2}.

In \cite{KMN2,KMN3}, we constructed the following generalization of Hall's Segal--Bargmann transform. Let $h$ be a strongly convex $Ad$-invariant function $h\in C^\infty(T^*K)$ and denote the vertical tangent bundle of $T^*K$ by $\mathcal{P}_0$. Then for each $\tau\in\mathbb{C}_+$,
\[
{\mathcal P}_\tau = e^{\tau {\mathcal L}_{X_h}} {\mathcal P}_0,
\]
is a purely complex positive integrable polarization of $T^*K$, which therefore induces a K\"ahler structure $(T^*K,\omega,J_\tau)$.

Let $\mathcal{H}_\tau$ be the half-form corrected K\"ahler quantization of $T^*K$ with respect to $\mathcal{P}_\tau.$ We show that the operator $e^{-\ii\tau \hat h}$ is a densely defined linear map from ${\mathcal H}_0$ to
${\mathcal H}_\tau$ which intertwines the canonical $K\times K$ actions, and moreover, there exists a densely defined operator $E(\tau,h)$ on $\mathcal{H}_0$ such that the composition
\[
e^{-\ii\tau\hat{h}}\circ E(\tau, h):\mathcal{H}_0\rightarrow\mathcal{H}_\tau
\]
is a unitary isomorphism.

The lesson here is that one might expect that the precomposition of a complex-time flow with some sort of ``Schr\"odinger'' quantization might yield a unitary isomorphism. In the remainder of this section, we investigate to what extent this intuition holds in the toric setting. We will see that, as in the case of $T^*K$, complex-time evolution in the toric setting is a sort of analytic continuation, but on the other hand, precomposing this analytic continuation with a Schr\"odinger-type quantum operator yields a map which is only asymptotically unitary.

\subsection{Starting at a toric K\"ahler polarization}
\label{kahlerstart}

Let $L\to X_{P}$ be the prequantum bundle. As in \cite{BFMN, KMN1}, we take the connection on $L$ to be
\[
\nabla{\mathbbm 1}^{U(1)} = \left(-\ii xd\theta\right)\,{\mathbbm 1}^{U(1)},
\]
where ${\mathbbm 1}^{U(1)}$ is a unitary trivialization of $L$ over $\check X_{P}$.

As in Section \ref{sstfp}, let $\psi$ be a smooth strongly convex function on $P$. The prequantum operator associated to $\psi$ is
\[
\hat\psi= \ii\nabla_{X_{\psi}} + \psi= \ii X_{\psi}- x\cdot\frac{\partial\psi
}{\partial x} + \psi.
\]

Consider the family of symplectic potentials on $X_{P}$ given by $g_{s}= g+s\psi, s>0$, where $g$ is some fixed symplectic potential. The half-form corrected quantization of $X_{P}$ in the polarization ${\Po}_{g_{s}}=e^{\ii s\mathcal{L}_{X_\psi}}\Po_g$ (Theorem \ref{proppolspsi}) is given by \cite{BFMN, KMN1}
\begin{equation}
\label{hilbertgs}
{\mathcal{H}}_{{\Po}_{g_{s}}} = \left\{  \sigma_{s}^{m}=w_{s}^{m}
e^{-h_{s}(x)} {\mathbbm 1}^{U(1)}\otimes\sqrt{dZ_{s}}, m\in P\cap{\mathbb{Z}}^{n} \right\}  ,
\end{equation}
where $h_{s}(x)=x\cdot\frac{\partial g_{s}}{\partial x}-g_{s}(x),$ $w_{s}^{j} = e^{y_{j}(s)+\ii\theta_{j}},$ and $y_{s}^{j}= \frac{\partial g_{s}}{\partial x^{j}}$. (Recall from Section \ref{sec:prelim} that $P\cap\mathbb{Z}^n = \check P\cap\mathbb{Z}^n$, see (\ref{halflambdas}) and (\ref{t11})).

From Theorem \ref{proppolspsi}, we have $e^{\ii s{\mathcal{L}}_{\psi}} dZ= dZ_{s}.$ As in \cite{KMN1}, we will define the action of the operator $e^{is{\mathcal{L}}_{\psi}}$ on half-forms in a way which is consistent with its action on the algebra of smooth functions on $X_{P}$ by setting
\begin{equation}
\label{evolsqrtdz}
e^{\ii s{\mathcal{L}}_{\psi}} \sqrt{dZ}= \sqrt{dZ_{s}}.
\end{equation}

\begin{proposition}
For all $s>0$, the operator $e^{s\hat\psi}\otimes e^{\ii s{\mathcal{L}}_{\psi}}: {\mathcal{H}}_{{\Po}_{g}} \to{\mathcal{H}}_{{\Po}_{g_{s}}}$
is an isomorphism, and
\[
e^{s\hat\psi} \otimes e^{\ii s{\mathcal{L}}_{\psi}}\sigma^{m}_{0} = \sigma^{m}_{s}, \,\, m\in P\cap{\mathbb{Z}}^{n}.
\]
\end{proposition}

\begin{proof}
We have
\[
e^{s\hat\psi} \left(  w^{m} e^{-h_{0}}\right)  = e^{-s(x\cdot\frac{\partial\psi}{\partial x}-\psi)} e^{-h_{0}} e^{\ii sX_{\psi}} \left(  w_{0}^{m}\right)  = e^{-h_{s}}w_{s}^{m}.
\]
The result then follows from (\ref{hilbertgs}) and (\ref{evolsqrtdz}).
\end{proof}

\bigskip
Recall from \cite{KMN1} that as $s\rightarrow\infty,$
\begin{equation}
\label{norms}
\|\sigma^{m}_{s}\|_{L^2(X_P,L)} \sim \pi^{\frac{n}{4}} e^{g_{s}(m)},\ m\in P\cap{\mathbb{Z}}^{n}.
\end{equation}
Therefore, the operator $e^{s\hat\psi}$, which evolves polarized sections along the flow of K\"ahler structures, is not unitary. This is an analog of what happens for the cotangent bundle $T^*K$ of a compact Lie group in \cite{KMN2,KMN3}. In particular, as reviewed in Section \ref{sec:preQops}, the corresponding operator is analytic continuation from $K$ to $K_{\mathbb{C}}$, which is of course not unitary. The unitarity of Hall's coherent state transform is achieved by precomposing with the heat kernel operator. In the setting of toric manifolds, we will see in Section \ref{subsec:startvert} that unitarity can only be achieved asymptotically as $s\to\infty$; this is expected from \cite{KMN1}.

Before dealing with unitarity, we must study the limit object. The quantum states for the quantization of $X_{P}$ in the real polarization ${\Po}_{\mathbb{R}}$ have support along the Bohr--Sommerfeld fibers $\mu^{-1}(P\cap{\mathbb{Z}}^{n})$ \cite{BFMN}. The following is proved in \cite{BFMN,KMN1}: the $\Po_\mathbb{R}$-polarized quantum Hilbert space is
\[
{\mathcal{H}}_{{\Po}_{\mathbb{R}}} = \{\delta^{m}\otimes\sqrt{dX}, m\in P\cap{\mathbb{Z}}^{n}\},
\]
where $\delta^{m}$ is a Dirac delta distributional section of $L$ supported on $\mu^{-1}(m)$ and $dX = dx^{1}\wedge\cdots\wedge dx^{n}$.

\begin{theorem}
\label{thm:KMN1thm}
\ \cite{KMN1}
As distributional half-form corrected sections of $L$,
\[
\lim_{s\to\infty} \frac{\sigma^{m}_{s}}{\|\sigma^{m}_{s}\|_{L^{2}(X_P,L)}} = 2^{n/2}\pi^{n/4} \delta^{m}\otimes\sqrt{dX}.
\]
\end{theorem}

Since the flow of $\psi$ preserves $\Po_\mathbb{R}$, a natural quantization of the classical observable $\psi$ on ${\mathcal{H}}_{{\Po}_{\mathbb{R}}}$ is given by the operator
\begin{align*}
\hat\psi_{\mathbb{{\mathbb{R}}}}: {\mathcal{H}}_{{\Po}_{\mathbb{R}}} &  \to{\mathcal{H}}_{{\Po}_{\mathbb{R}}}\\
\delta^{m}\otimes\sqrt{dX}  &  \mapsto\psi(m) \delta^{m}\otimes\sqrt{dX},
\end{align*}
since the support of the distributional section $\delta^{m}\otimes\sqrt{dX}$ is $\mu^{-1}(m)$. We extend the domain of this operator by setting
\begin{align}
\label{eqn:deltam}
\nonumber
\hat\psi_{\mathbb{{\mathbb{R}}}}: {\mathcal{H}}_{{\Po}_{g}} & \to{\mathcal{H}}_{{\Po}_{g}}\\
\sigma^{m}  &  \mapsto\psi(m)\sigma^{m}.
\end{align}

We then have an (asymptotic) analog of the coherent state transform of Hall in the context of toric manifolds.
Define an operator $A^{\psi}_{g,s}:\mathcal{H}_{\Po_{g}}\rightarrow\mathcal{H}_{\Po_{g_s}}$ by
\begin{equation}
\label{eqn:Ags}
A^\psi_{g,s}:=\left( e^{s\hat\psi}\otimes e^{\ii s{\mathcal{L}}_{\psi}}\right)
\circ e^{-s\hat\psi_{\mathbb{{\mathbb{R}}}}}.
\end{equation}
In fact, the map $A^\psi_{g,s}$ is an analog of the KSH map of \cite{KMN2,KMN3}. For each $s$, identify $\mathcal{H}_{\Po_{g_s}}\simeq\mathbb{C}^{\#(P\cap\mathbb{Z}^n)}$ via the  basis of $L^2$-normalized holomorphic sections $\{\sigma^m_s/\|\sigma^m_s\|_{L^2(X_P,L)}\},$ and for $\mathcal{H}_{\Po_\mathbb{R}}$ via the basis $\{\delta^m\otimes\sqrt{dX}\}.$

\begin{theorem}
\label{thm:Ags}
The map $A^\psi_{g,\infty}:\mathcal{H}_{\Po_g}\rightarrow\mathcal{H}_{\Po_\mathbb{R}}$
determined by
\[
A^\psi_{g,\infty}\left(\frac{\sigma_{0}^{m}}{\|\sigma^m_{0}\|_{L^{2}(X_P,L)}}\right) :=\frac{(2\pi)^{n/2}e^{g(m)}}{\|\sigma^m_0\|_{L^2(X_P,L)}}\delta^{m}\otimes\sqrt{dX}
\]
satisfies $\lim_{s\rightarrow\infty}A^\psi_{g,s}=A^\psi_{g,\infty}$ in $GL(\#(P\cap\mathbb{Z}^n) ,\mathbb{C}).$
\end{theorem}

\begin{proof}
We can compute
\begin{align*}
\lim_{s\to\infty}\left( e^{s\hat\psi}\otimes e^{\ii s{\mathcal{L}}_{\psi}}\right)
\circ e^{-s\hat\psi_{\mathbb{{\mathbb{R}}}}}
\left(  \frac{\sigma_{0}^{m}}{\|\sigma^m_{0}\|_{L^{2}(X_P,L)}}\right) &= \frac{e^{-s\psi(m)}}{\|\sigma^m_0\|_{L^2(X_P,L)}} \sigma_{s}^{m}.
\end{align*}
so the result follows from (\ref{norms}) and Theorem \ref{thm:KMN1thm}.
\end{proof}

\subsection{Starting at the vertical polarization}
\label{subsec:startvert}

Recall that the quantization of $T^*\T^n$ in the vertical polarization ${\mathcal P}_{v} = \operatorname{span}_\mathbb{C}\left\{\frac{\partial}{\partial x^j}\right\}_{j=1,\dots,n},$ where we identify $\check X_P \cong \T^n\times \check P \subset (\C^\times)^n$, is given simply by
\[
{\mathcal H}_{{\mathcal P}_{v}}=L^2(\T^n),
\]
which is generated by the monomials $\{e^{\ii m\theta}\}_{m\in \Z^n}.$

In this section, we shall see that some of these monomials, namely the ones corresponding to integral points inside the moment polytope $P$, produce the monomial holomorphic sections of $L\to X_P$ in the K\"ahler quantization of $(X_P,\omega,I_g)$ when acted upon by the time-$\ii$ Hamiltonian evolution of the symplectic potential $g$, considered as a function on $X_P$.
Note that the Hamiltonian function used in the complexification process, the symplectic potential $g$, is not smooth on $X_P$ (compare to \cite{KMN2}): it is singular along $\mu^{-1}(\partial P)$. This allows the Legendre transform $y=\frac{\partial g}{\partial x}$ to generate all of $\R^n$ from the bounded region $\check P$.

\begin{theorem}
\label{timeiflowmonomials}
The time-$\ii$ flow of the vector field $X_{g}$ induces an isomorphism from the algebra $\mathcal{A}$ generated by the characters $e^{\ii m\theta}, m\in{\mathbb{Z}}^{n}$, of ${\mathbb{T}}^{n}$ onto the algebra generated by $I_{g}$-monomial (meromorphic) functions on $(X_{P}, I_{g})$
\begin{align*}
\exp(\ii X_{h}) &: \mathcal{A}\rightarrow\mathcal{O}_{I_g}(\check X_P)\\
e^{\ii m \theta} &\mapsto w^{m}.
\end{align*}
\end{theorem}

\begin{proof}
We have from (\ref{hamg}) that for any $m\in \Z^n$, $e^{\ii X_g} e^{im\theta} = e^ {m(y+\ii \theta)} = w^m.$
\end{proof}

\bigskip

The prequantum operator associated to $g$ is
\begin{equation}
\label{prequantumg}
\hat g = \ii \nabla_{X_g} + g = \ii X_g -h,
\end{equation}
where $h = x\cdot y -g$ is the K\"ahler potential for $(X_P,\omega,I_g)$.

Consider the following subset of the set of characters of $\T^n$, which is a finite-dimensional subspace of the Hilbert space for the quantization of $T^*\T^n$ in the vertical polarization:
\[
{\mathcal A}_P = \operatorname{span}_\mathbb{C} \left\{e^{\ii m\theta}, m\in P\cap\Z^n\right\}.
\]

Let $\Theta = d\theta_1\wedge\cdots\wedge d\theta_n.$

\begin{theorem}
The operator
\[
e^{\hat g}\otimes e^{\ii {\mathcal L}_{X_g}}: {\mathcal A}_P {\mathbbm 1}^{U(1)}\otimes \sqrt{d\Theta} \to {\mathcal H}_{{\mathcal P}_g}
\]
is an isomorphism.
\end{theorem}

\begin{proof}
The result follows from (\ref{hilbertgs}), Proposition \ref{timeiflowprop} $(b)$ and  Theorem \ref{timeiflowmonomials} since
\[
e^{\hat g}\otimes e^{i{\mathcal L}_{X_g}} \left(e^{\ii m\theta} {\mathbbm 1}^{U(1)}\otimes \sqrt{d\Theta}\right) = w^m e^{-h} {\mathbbm 1}^{U(1)} \otimes \sqrt{dZ} = \sigma^m,
\]
for all $m\in P\cap \Z^n$.
\end{proof}

In order to study the unitarity properties of this isomorphism, we define a quantum operator $\hat g_\R:{\mathcal A}_P {\mathbbm 1}^{U(1)}\otimes \sqrt{d\Theta}\to {\mathcal A}_P {\mathbbm 1}^{U(1)}\otimes \sqrt{d\Theta}$ by setting
\[
\hat g_\R \left(e^{\ii m\theta}{\mathbbm 1}^{U(1)}\otimes \sqrt{d\Theta}\right) = g(m) e^{\ii m\theta}{\mathbbm 1}^{U(1)}\otimes \sqrt{d\Theta}.
\]
This is analogous to what we did in \ref{eqn:deltam}.

As in \cite[eq. 3.12]{FMMN2}, define an inner product in $\mathcal{A}_{\Po}$ such that $\|e^{im\theta}\otimes\sqrt{d\Theta}\|=\pi^{n/4}.$
As in (\ref{eqn:Ags}), consider the operator $A_{0,1}^{g_s}:{\mathcal A}_P {\mathbbm 1}^{U(1)}\otimes \sqrt{d\Theta} \rightarrow {\mathcal H}_{{\mathcal P}_{g_s}}$ defined by
\[
A_{0,1}^{g_s}:=\left(e^{\hat g_s} \otimes e^{\ii{\mathcal L}_{X_{g_s}}}\right) \circ e^{-\widehat{(g_s)}_\R}.
\]
The following corollary follows from (\ref{norms}) and Theorem \ref{thm:KMN1thm}.
\begin{theorem}
There exists a unitary operator $U:{\mathcal A}_P {\mathbbm 1}^{U(1)}\otimes \sqrt{d\Theta} \to {\mathcal H}_{{\mathcal P}_{\mathbb{R}}}$ such that $U=\lim_{s\rightarrow \infty} A_{0,1}^{g_s},$ where the limit is understood as in Theorem \ref{thm:Ags}.
\end{theorem}

\begin{remarks}
\begin{enumerate}
\item The map $A^g_{0,s}=A^{sg}_{0,1}$ for $s\neq1$ would map ${\mathcal A}_P {\mathbbm 1}^{U(1)}\otimes \sqrt{d\Theta}$ to the quantization of $X_P$ with respect to a K\"ahler polarization corresponding to a metric with cone angles, although we do not study these more general maps here.

\item The results in this section show that the quantum Hilbert spaces for K\"ahler quantizations of toric manifolds of dimension $2n$ can all be obtained by complex-time evolution from appropriate finite-dimensional subspaces of the Hilbert space for the Schr\"odinger quantization of the cotangent bundle $T^{*}{\mathbb{T}}^{n}$.
\end{enumerate}
\end{remarks}

\section{Singular indefinite K\"{a}hler polarizations}
\label{sssik}

\subsection{{Geometry of indefinite K\"{a}hler polarizations}}

In this section, we consider again the geodesic ray of K\"ahler structures on $X_P$ induced by the symplectic potentials $g_{s}:=g_{P}+s\psi$, where $\psi\in C^\infty_+(P)$. The results of this section, however, hold \emph{mutatis mutandi} for general $\psi\in C^{\infty}(P)$. For all $s>0$, $g_{s}$ induces a K\"{a}hler structure on $X_P$ in the usual way (\ref{eqn:KahlerStruct}). Our task here is to investigate the behavior of $g_{-s}$, as well as the resulting K\"{a}hler-type geometry of $X_P$, for both finite $s>0$ and for large $s>0$.

First, observe that for $s$ sufficiently large, $g_{-s}$ will no longer be convex on $\check P$. This has several consequences. Let $s^{cvx}\geq0$ be the smallest real number such that $g_{s}$ is convex for all $s>-s^{cvx}$.

To simplify notation in this section, we denote the Hessian of the symplectic potential $g_s$ by $H_s:=H_{g_s}.$ Then for all $s\leq-s^{cvx}$, the zero locus $Z_{s}:=\{x\in P:\det H_{s}(x)=0\}$ of the Hessian is nonempty. The set of possible signatures of $n\times n$ matrices with nonzero determinant is $\{(a,b):a+b=n\}.$ We may decompose $P\setminus Z_{s}$ into a union of open sets
\[
P\setminus Z_{s}= \bigcup\limits_{a+b=n}P_s^{(a,b)},
\]
where $P^{(a,b)}_s:=\{x\in P:\operatorname{sig}(H_{s}(x))=(a,b)\}$. Each open set $P^{(a,b)}_s$ can be written as a disjoint union of open connected neighborhoods, which we call \emph{islands}:
\[
P_s^{(a,b)}=\bigcup_{j=1}^{n_s^{(a,b)}}U_{s,j}^{(a,b)}.
\]
It is possible that $n_s^{(a,b)}=\infty$ (perhaps even uncountably infinite, but we will keep our notation simple and assume at most a countable number of islands); we will give an example below.

As is clear from (\ref{eqn:KahlerStruct}), the K\"{a}hler geometry of $X_P$ is controlled by the Hessian $H_{s}$ of $g_{s}$, whose behavior is essentially determined by the signature, a constant on each island $U_{s,j}^{(a,b)}$.

For $s>-s^{cvx}$, it is well known that the resulting path of metrics $s\in(-s^{cvx},\infty)\mapsto\gamma_{s}$ on $X_P$ is a Mabuchi geodesic in the space of K\"{a}hler metrics in the class of $\omega$. The constant $-s^{cvx}$ is what Rubinstein and Zelditch call the convex lifespan of the geodesic \cite{Rubinstein-Zelditch1,Rubinstein-Zelditch2}. They show that for $s<-s^{cvx}$, the path $s\mapsto\gamma_{s}$ fails to be even a weak solution to the geodesic equation and that as $s$ approaches $-s^{cvx}$ from above, the metric develops singularities\footnote{Rubinstein and Zelditch actually parameterize $t=-s$.}. On the other hand, they show that the metric remains positive definite on a dense set (the complement of the singular locus) and that on this dense set, the path $\gamma_{t}$ still solves the geodesic equation (actually, they work in the equivalent formulation in terms of the homogeneous Monge--Amp\`{e}re equation).

For all $s$ there are positive islands (that is, islands where the signature of $H_s$ is positive definite) around each vertex. As $s$ approaches the convex lifespan, the metric becomes singular. The analysis of Rubinstein--Zelditch, from the fixed-symplectic point of view, cuts out a neighborhood of the union of the indefinite islands on $X_P$ and glues the result, thus yielding a toric manifold with a singular metric which, outside the singular set, is biholomorphic to the original $X_P$. It is, however, interesting to investigate $X_P$ in the indefinite region at both the level of geometry and in terms of the behavior of the quantization of $X_P$.

\bigskip
We begin with some remarks on the geometry of the indefinite islands. First, because we consider strongly convex $\psi$, we can show that there is some minimal $s^{neg}\geq s^{cvx}$ such that for $s<-s^{neg}$, there is a family of connected negative-definite islands $U_{s}^{(0,n)}$ which ``fills'' $P$ as $s\rightarrow-\infty$. On the other hand, for each vertex $v\in P$ and for all $s$ there is a positive-definite island $U_{s}^{v,(n,0)}$ containing $v$. After proving these two results, we will give examples which show that these results are the best one can expect.

We say that a family of islands $\{U_{s}^{(a,b)}:s\in \mathcal{I}\subset\mathbb{R}\}$ is \emph{increasing} if $U_{s}^{\varepsilon}\supset U_{s^{\prime}}^{\varepsilon}$ whenever $s> s^{\prime}.$

\begin{theorem}
Let $K$ be a compact subset of the interior $\check{P}$ of the moment polytope of $X_P.$ Then there exists some $s_{K}$ and an increasing family $\{U_{-s}^{K,(0,n)}:s\in(s_{K},\infty)\}$ of islands such that $K\subset U_{s}^{K,(0,n)}$ for all $s\,<-s_{K}$. Moreover, for any compact subset $K\subset\check{P}$,
\[
\bigcup_{s<-s_{K}}U_{s}^{K,(0,n)}=\check{P}.
\]
\end{theorem}

\begin{proof}
Since Hessian of $\psi$ is bounded on $P$ and hence on $K$, for each point in $K$ there is some $s(x)$ such that $H_{-s(x)}(x)<0.$ Since $H_{-s(x)}<0$ is an open condition, there is a an open neighborhood $U_{-s(x)}^{x,(0,n)}$ of $x$ on which $H_{-s(x)}<0$. Moreover, for any $s<-s(x),$ $H_{s}(x)$ remains negative definite on $U_{-s(x)}^{x,(0,n)}$. The compact set $K$ is necessarily covered by such open neighborhoods. Choose a finite subcover and let $s_{K}$ be the largest $s(x)$ which occurs in this finite subcover. Then for any $s>s_{K}$, $H_{-s}(x)<0$ for all $x\in K$.

It follows that for any $K$, there is a negative-definite island, which we denote by $U_{s_{K}}^{K,(0,n)}$, containing $K$. Moreover, if $K_{1}$ and $K_{2}$ are nondisjoint compact subsets of $\check{P}_{X}$ and $s_{K_{1},K_{2}}:=\max\{s_{K_{1}},s_{K_{2}}\}$, then $U_{s_{K_{1},K_{2}}}^{K_{1}\cup K_{2},(0,n)}$ contains $K_{1},K_{2},U_{s_{K_{1}}}^{K_{1},(0,n)}$ and $U_{s_{K_{2}}}^{K_{2},(0,n)}$. That is, $U_{s_{K_{1},K_{2}}}^{K_{1},(0,n)} =U_{s_{K_{1},K_{2}}}^{K_{2},(0,n)}=U_{s_{K_{1},K_{2}}}^{K_{1}\cup K_{2},(0,n)}.$ To obtain the desired family, apply this to an increasing sequence $K_{1}\subset K_{2}\subset\cdots$ of compact subsets of $\check{P}$ such that $\bigcup K_{j}=\check{P}$.
\end{proof}

\begin{theorem}
For each vertex $v\in P$ there is a family of islands $\{U_{s}^{v,(n,0)}:s\in\mathbb{R}\}$ such that for each $s\in\mathbb{R}$, $v\in U_{s}^{v,(n,0)}.$
\end{theorem}

\begin{proof}
The Hessian of $g_{P}$ is $\left(\tfrac12\sum_k \frac{\nu_k^i\nu_k^j}{\ell_{k}(x)}\right)$ (see equation (\ref{t11})), where we order the terms so that $\ell_k(v)=0$ for $k=1,\dots,r$. Hence, we see that the Hessian of $g_P$ is unbounded near $v$. On the other hand, the Hessian of $\psi$ is bounded on $P$. Hence, for any $s$, $H_{s}$ is positive definite for $x$ sufficiently close to $v$.
\end{proof}

\bigskip
In the next example, we see the positive-definite islands around the vertices of $P$ shrinking as $s$ becomes large and negative. Moreover, $U_{s}^{v,(n,0)}\cap U_{s}^{v^{\prime},(n,0)}=\varnothing$ for all $s<-2$ unless $v=v^{\prime}$.

\begin{example}
Consider $\mathbb{C}P^{1}\times\mathbb{C}P^{1}$ with moment polytope $P=[0,1]\times[0,1]$ and symplectic potential $g_{s}=\frac{1}{2}\sum_{j=1}^{2}(x_{j}\log x_{j}+(1-x_{j})\log(1-x_{j}))+s(2x_{1}^{2}+\,x_{2}^{2})/2.$ Then $s^{cvx}=1$ and $s^{neg}=2$. For all $s<-s^{neg}$, the island decomposition of $P$ is

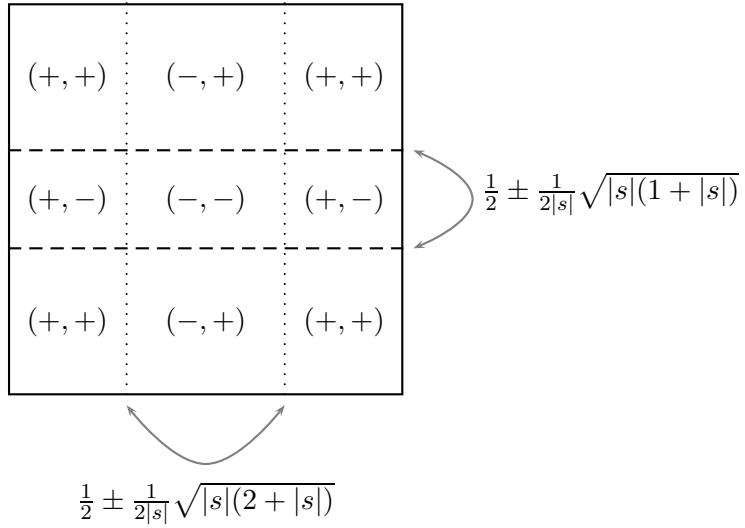
\begin{pspicture*}(-7,-4.9)(7,3.5)
%\psgrid
\psset{unit=1.3cm}
\psframe(-2,-2)(2,2)
\psline[linestyle=dashed](-2,-0.5)(2,-0.5)
\psline[linestyle=dashed](-2,0.5)(2,0.5)
\psline[linestyle=dotted](-0.8,-2)(-0.8,2)
\psline[linestyle=dotted](0.8,-2)(0.8,2)
\pscurve[linecolor=gray]{<->}(2.1,0.5)(2.7,0)(2.1,-0.5)
\pscurve[linecolor=gray]{<->}(-0.8,-2.1)(0,-2.7)(0.8,-2.1)
\rput[Bl](2.8,0){$\tfrac12 \pm\tfrac{1}{2|s|}\sqrt{|s|(1+|s|)}$}
\rput(0,-3.1){$\tfrac12 \pm\tfrac{1}{2|s|}\sqrt{|s|(2+|s|)}$}
\rput(0,0){$(-,-)$}
\rput(-1.4,0){$(+,-)$}
\rput(1.4,0){$(+,-)$}
\rput(0,-1.25){$(-,+)$}
\rput(0,1.25){$(-,+)$}
\rput(-1.4,1.25){$(+,+)$}
\rput(-1.4,-1.25){$(+,+)$}
\rput(1.4,1.25){$(+,+)$}
\rput(1.4,-1.25){$(+,+)$}
\end{pspicture*}
\end{example}

Next, we give an example with an infinite number of negative-definite islands for a fixed value of $s$ in any neighborhood of $x=1/2$, and a little thought shows that one can arrange a sequence of such ``bad'' points which converges to the boundary of $P$ and thus yields a countably infinite number of negative islands for a countably infinite number of values of $s$.

\begin{example}
Consider $\mathbb{C}P^{1}$ with moment polytope $P=[0,1]$ and symplectic potential $g_{s}=\frac{1}{2}(x\log x+(1-x)\log(1-x))+s\psi(x)$, where $\psi(x)$ is such that $\psi^{\prime\prime}(x)=2+e^{-1/(x-1/2)^{2}}\sin\left(  \frac{1}{x-1/2}\right)  $. One can verify that $\psi^{\prime \prime}(x)>1$ on $[0,1]$ and hence that it is strongly convex.

At $s=-1$, we have $g_{-1}^{\prime\prime}(1/2)=0$ and the oscillatory behavior of $\sin\left(  \frac{1}{x-1/2}\right)  $ shows that there are an infinite number of open intervals on which $g_{-1}^{\prime\prime}(x)>0$ interspersed by intervals on which $g_{-1}^{\prime\prime}(x)<0$, that is, there are an infinite number of positive and negative islands in any neighborhood of $x=1/2$ at $s=-1$. For any $s<-1,$ there are only a finite number of positive-definite and negative-definite islands on $P,$ but one can arrange $\psi$ to consist of a convergent sum of such oscillatory terms centered on a sequence, say $x_{n}=1/2^{n}$, of ``bad'' points. There will then be a decreasing sequence $s_{1}=-1>s_{2}>\cdots>s_{n}>\cdots$ such that $g_{s_{n}}^{\prime\prime}(x_{n})=0$ and at $s=s_{n}$ there are an infinite number of positive and negative islands in any neighborhood of $x_{n}$.

Note that if a point $x$ is in a negative-definite island at time $s$, it is in a negative-definite island for all smaller $s$. Thus, the negative islands do not ``move'', they just grow.
\end{example}

\bigskip
We now turn to the complex structure on $X$. We will denote the polarization $\Po_{g_s}$  by $\Po^s$; it is given in terms of the decomposition $T_{(x,\theta)}X\simeq T_{x}P\oplus T_{\theta}\mathbb{T}^{n}\simeq\mathbb{R}^{n} \oplus\mathbb{R}^{n}$ by the following lemma.

\begin{lemma}
\label{lemma:P}
For $x\in \check P$ and $\theta\in\T^n$, $\Po_{(x,\theta)}^{s}=\left\langle
\begin{pmatrix}
\mathbf{1}\\
-\ii H_{s}(x)
\end{pmatrix}
v:v\in\mathbb{R}^{n}\right\rangle .$
\end{lemma}

\begin{proof}
The $(0,1)$-tangent bundle $\Po^s$ is the kernel of
\[
\mathbf{1}+\ii I_{g_s}=\begin{pmatrix}\mathbf{1}&-\ii H^{-1}_s\\ \ii H_s&\mathbf{1}\end{pmatrix},
\]
which is clearly equal to the claimed subspace.
\end{proof}

\bigskip
For $s\leq-s^{cvx}$, the Hessian $H_{s}(x)$ is no longer necessarily positive definite and hence no longer defines a positive complex structure at every point, but the polarization $\Po^s$ is nevertheless well defined (it is a smooth complex integrable Lagrangian distribution). For $x\in P$ such that $\det H_{s}(x)\neq0$, $\Po^s_{(x,\theta)}$ is the $(0,1)$-tangent space of a complex structure which is not necessarily positive with respect to $\omega$. At such points, the induced ``metric'' $\omega_{(x,\theta)}(\cdot,J_{(x,\theta)}\cdot)$ is an indefinite bilinear symmetric tensor whose signature is determined by the signature of $H_{s}(x)$ via (\ref{eqn:KahlerStruct}).

At a point $x_{0}\in P$ where $\det H_{s}(x_{0})=0$, the next lemma shows that the polarization $\Po_{(x_{0},\theta)}^{s}$ is no longer purely complex and therefore does not define even an indefinite metric; rather, the (possibly indefinite) metric $\omega(\cdot,J\cdot)$ becomes singular as $x$ approaches $x_{0}$ from an indefinite island.

\begin{lemma}
$\overline{\Po_{(x,\theta)}^s}\cap \Po_{(x,\theta)}^s=\left(\ker H(x)\right)^{\mathbb{C}}\oplus\{0\}.$
\end{lemma}

\begin{proof}
By Lemma \ref{lemma:P}, if $v\in\ker H(x)$, then $\binom{\mathbf{1}}{-\ii H(x)}v=\binom{v}{0}\in\overline{\Po^{s}_{(x,\theta)}}\cap \Po^{s}_{(x,\theta)}$. To show the opposite inclusion, suppose $v,w\in\mathbb{R}^{n}$ are such that $\overline{\binom{\mathbf{1}}{-\ii H(x)}v}=\binom{\mathbf{1}}{-\ii H(x)}w\in \mathcal{P}^s_{(x,\theta)}$. Then $\binom{v}{\ii H(x)v}=\binom{w}{-\ii H(x)w}$. Comparing real parts shows $w=v$, and comparing imaginary parts shows that $H(x)v=0$, whence $\overline{\binom{\mathbf{1}}{-\ii H(x)}v}=\binom{\mathbf{1}}{-\ii H(x)}w=\binom{v}{0}\in\ker H(x)\oplus\{0\}$.
\end{proof}

\bigskip
We conclude this section with short a discussion of the metric behavior of $X_P$ for $s$ slightly smaller than $-s^{cvx}.$ Fix $x_{0}\in Z$ and let $v\in\ker H_{s}(x_{0}) $. Then for small enough $\varepsilon>0$, the path $x(t):=x_{0}+tv,~t\in(0,\varepsilon)$ is in some island $U$. Fix $\theta\in\mathbb{T}^{n}.$ As an element of the tangent space $T_{(x(t),\theta)}X$, the length of $v$ is $\gamma^{s}(v,v)=\,^{t}vH_{s}(x(t))v$ and hence goes to zero as $t\rightarrow0$. On the other hand, the length of the path $(x(t),\theta ),~t\in(0,\varepsilon)$ is $\int_{0}^{\varepsilon }\gamma_{(x(t),\theta)}^{s}(v,v)dt=\,^{t}v\left(  \int_{0}^{\varepsilon} H_{s}(x(t))dt\right)  v$, which is some (finite) nonzero but possibly negative constant, depending on the signature of $H_{s}$ on $U$.

For $x_{0},$ $v$ and $\varepsilon$ as above and fixed $t\in (0,\varepsilon) $, the  circle $\{(x(t),t^{\prime }v):t^{\prime}\in\mathbb{R}\}\subset\{x(t)\}\times\mathbb{T}^{n}$ has circumference $2\pi\,^{t}vH_{s}^{-1}(x(t))v.$ Expanding $v$ in terms of an orthonormal eigenbasis of $H(x(t))$ and using the fact that $v\in\ker H(x(0))$ shows that $2\pi\,^{t}vH_{s}^{-1}(x(t))v\rightarrow\pm\infty$ as $t\rightarrow0$. That is, the angular size of the circle over $x(t)=x_{0}+tv$ which is symplectically dual to~$v$ grows without bound as $x(t)$ approaches $x_{0}.$

Rubinstein and Zelditch encode this observation, which is already easy to see from (\ref{eqn:KahlerStruct}), in terms of the Monge--Amp\'{e}re measure as follows. For $g_{s}$ convex, the symplectic form $\omega=dx^{j}\wedge d\theta_{j}$ can be written in terms of the Legendre dual variable $y_{s}=\partial g_{s}/\partial_{x}$ as $H_{s}^{-1}dy_{s}^{j}\wedge d\theta_{j}$. The Liouville measure $(-1)^{n(n-1)/2}\omega^{n}/n!$ then becomes $MA(g_{s})\wedge d\theta_{1}\wedge\cdots\wedge d\theta_{n},$ where $MA(g_{s})$ is the \emph{Monge--Amp\'{e}re measure} $MA(g_{s}):=\det H_{s}^{-1} dy_{s}^{1}\wedge\cdots\wedge dy_{s}^{n}$ on $\mathbb{R}^{n}$ (see equation (12) \emph{ff.} in \cite{Rubinstein-Zelditch1}). Rubinstein and Zelditch show in \cite[Theorem 1]{Rubinstein-Zelditch2} that the Monge--Amp\`{e}re measure at the convex lifespan charges the singular set $Z_{-s^{cvx}}$ with positive mass:
\[
\int_{[0,-s^{cvx}]\times\mathbb{R}^{n}}ds\wedge MA(g_{s})=\int_{Z_{-s^{cvx}}}ds\wedge MA(g_{-s^{cvx}})>0.
\]
That is, as $s\rightarrow-s^{cvx}$, $MA(g_{s})$ converges to a positive current supported on the zero set of $H_{s}$.

\subsection{Polarized sections in the negative-definite island}
\label{ssqhs}

As always, we assume that $X$ is equipped with a symplectic potential $g_{s}=g_{P}+s\psi,$ where $\psi\in C^\infty_+(P)$, and we write $y_{s}:=\partial g_{s}/\partial x$ and denote the Hessian of $g_s$ at the point $x$ by $H_{s}(x)$. Recall from the last section that we decompose $P\setminus Z_{s}$, where $Z_{s}:=\{x\in
P:\det H_{s}(x)=0\}$, into a union of connected open sets
\[
P\setminus Z_{s}=\bigcup_{a+b=n} \bigcup_{j=1}^{n_s^{(a,b)}}U_{s,j}^{(a,b)},
\]
where $H_{s}(x)$ has constant signature $(a,b)$ on the island $U_{s,j}^{(a,b)}.$ Moreover, there is an $s^{neg}>0$ and a family of negative-definite islands $\{U_{s}^{(0,n)}:s\in(-s^{neg},\infty)\}$ which is increasing and fills $\check{P}$ as $s\rightarrow-\infty$.

For $s>-s^{cvx}$, the symplectic potential $g_{s}$ induces a positive K\"{a}hler structure on all of $X_P$, and the (geometric) quantum Hilbert space---defined to be the space of holomorphic sections of a prequantum line bundle $L$ over $X_P$---is spanned by the monomials $\sigma_{s}^{m}$ associated to integer points~$m\in P\cap\mathbb{Z}^{n}$. (In this section, we do not incorporate the half-form correction, so $L$ is a line bundle with Chern class $[\omega/2\pi].$) On the open orbit with respect to a unit-norm section $\mathbbm{1}^{U(1)}\in\Gamma_{\check X_P}(L)$, the section $\sigma_{s}^{m}$ can be written $w_{s}^{m}e^{-(xy_{s}-g_{s})} \mathbbm{1}^{U(1)}$, where $w_{s}=e^{y_{s}+\ii \theta}$.

Our first result is that even for $s<-s^{cvx}$, the vector space of $\Po^{s}$-polarized sections of $L$ is spanned by the monomials $\sigma_{s}^{m},~m\in P\cap\mathbb{Z}^{n}$.

\begin{theorem}
The dimension $\dim\ker\left.  \bar{\partial}_{s}\right\vert _{\Gamma(L)}$ is equal to $\#\left(  \mathbb{Z}^{n}\cap P\right)$. In particular, the set $\{\sigma^m: m\in P\cap\mathbb{Z}^n\}$ is a basis for $\ker\left.  \bar{\partial}_{s}\right\vert _{\Gamma(L)}$, where
\begin{equation}
\label{eqn:monomials}
\left.  \sigma^{m}\right\vert _{\check X_P}=e^{m\cdot(y+\ii\theta)}e^{-(x\cdot y-g)}\mathbbm{1}^{U(1)}.
\end{equation}

\end{theorem}

\begin{proof}
In the dense open subset $\check X_P\setminus Z_s$, solutions of $\bar\partial_s \sigma=0$ satisfy Cauchy--Riemann equations and a basis of the infinite-dimensional space of solutions is given by the monomial sections (\ref{eqn:monomials}). Since there is a neighborhood of each vertex which is biholomorphic to a neighborhood of $0$ in $\mathbb{C}^n$, the problem of finding those local solutions which extend to $X_P$ becomes equivalent to the standard construction of holomorphic sections on toric varieties (see (\ref{hzero}) \textit{ff.}). For a point in $Z_s\cap\check P$, certain components of $J$ blow up; for example, $(1+iJ)\partial_{\theta_j} =\partial_{\theta_j}-\ii\left(H^{-1}\right)^{jk}\partial_{x^k}$. But since $H$ is invertible away from the zero locus of $\det H$, a section $u$ satisfies $\nabla_{(1+\ii J)\partial_{\theta_j}}u=0$ for all $j=1,\dots,n$ if and only if it satisfies $\nabla_{(1+\ii J)H_{jk}\partial_{\theta_k}}u$ for all $j$, and the vector field $(1+\ii J)H_{jk}\partial_{\theta_k} =H_{jk}\partial_{\theta_k}-\ii H_{jk}\left(H^{-1}\right)^{kl}\partial_{x^l} =H_{jk}\partial_{\theta_k}-\ii\partial_{x^l}$ is well behaved near $\{\det H=0\}$. For a point in the intersection of the zero locus $Z_s$ with the relative interior of a codimension $p$ face of $P$, the kernel of $\bar\partial_s$ is the intersection of the kernel of $\bar\partial_s$ in the directions transverse to the inverse image of the face with respect to $\mu$ and the kernel in the longitudinal directions. In the transverse directions, one has the usual Cauchy--Riemann conditions. In the longitudinal directions, one may have to rescale some of the $(0,1)$ vector fields as above.
\end{proof}

\bigskip
Next, we examine how $\Po^{s}$-polarized sections behave in indefinite regions in $X_P$. Our first result, Theorem \ref{thm:ptwiseNorm}, shows that for a polarized section associated to an integral point in the polytope which lies in the negative-definite region, the pointwise norm of the section increases monotonically as one moves from the integral point toward the boundary of the negative-definite region. Theorem \ref{thm1} then shows that for large negative $s$, the $L^2$-norm of a polarized section is concentrated outside of any compact set in the negative-definite region.

\begin{theorem}
\label{thm:ptwiseNorm}
Let $s<-s^{neg}$ and suppose $m\in P\cap\mathbb{Z}^n$ is in an island in $P^{(0,n)}_s$. Then the pointwise norm of $\sigma_{s}^{m}$ (which is $\T^n$-invariant) is minimized at $m$ and strictly increasing on each ray emanating from $m$ and ending on the boundary of the negative-definite island containing $m$.
\end{theorem}

\begin{proof}
We have from (\ref{hilbertgs}) that $\left\vert \sigma_{s}^{m}\right\vert ^{2}=e^{2(m-x)\cdot y_{s}+2g_{s}}$ whence
\[
\operatorname{grad}\left\vert \sigma_{s}^{m}\right\vert ^{2}=e^{2(m-x)\cdot y_{s}+2g_{s}}2(m-x)\cdot H_{g_{s}}.
\]
Integrating this along the path $\gamma(t)=m+t(x-m),$ $t\in[0,1]$ yields that the gradient of the norm of the
holomorphic section $\sigma_{s}^{m}$ is
\[
\left| \sigma_{s}^{m}\right| ^{2}(x)=\left|\sigma_{s}^{m}\right|^{2}(m)-\int_{0}^{1}e^{2t(m-x)\cdot y_{s}(\gamma(t))+2g_{s}(\gamma
(t))}2t(m-x)\cdot H_{s}(\gamma_{t})\cdot(m-x)dt.
\]
The result then follows since $H_s$ is negative definite as long as $\gamma_{t}$ is in the negative-definite island.
\end{proof}

\bigskip
Let $\phi(x):=(x-m)\cdot\frac{\partial \psi}{\partial x}-\psi$.

\begin{lemma}
Fix $m\in P\cap\mathbb{Z}^n$ and $s_0<-s^{cvx}$. Then for each $s<s_0$ and each compact set $K\subset U_s^{(0,n)}$, we have
\begin{enumerate}
\item $|\sigma^m_s|(x,\theta)=e^{(m-x)\cdot y_0 +g_0} e^{-|s|\phi(x)},$ and
\item the function $\phi(x)$ achieves its maximum on $\partial K$.
\end{enumerate}
\end{lemma}

\begin{proof}$\quad$

\emph{1.)} follows from equation (\ref{hilbertgs}).

\emph{2.)} follows from an argument similar to that of Theorem \ref{thm:ptwiseNorm} (see also the proof of \cite[Lemma 3.7]{BFMN}, in particular the first two equations on p. 431).
\end{proof}

\begin{theorem}
\label{thm1}
For each $s_{0}<-s^{cvx}$ and for each compact set $K\subset U_{s_0}^{(0,n)}$ with smooth boundary $\partial K$ such that $\phi$ achieves a unique nondegenerate maximum on $\partial K$, there exist $\alpha,C>0$ such that for any $s\ll s_0$ and for any section $u\in H^0_s(X_{P},L)$,
\begin{equation}
\left\Vert u\right\Vert _{L^{2}(K,L)}^{2} < C e^{-|s|\alpha} \|u\|^2_{L^2(X_{P},L)}.
\label{eqn:1}
\end{equation}
\end{theorem}

\begin{remarks}
\begin{enumerate}
\item The constants $C$ and $\alpha$ depend on $s_0,\ K,\ \psi$ and $X_{P}$, but \emph{not} on $s$ or $u$.
\item The theorem holds in much more generality, but we prove only the simplest case here. Indeed, if $\phi$ achieves a degenerate minimum on $\partial K$, or if $K$ is not smooth at the minimum (for example, if the minimum occurs on a corner), then one actually obtains a larger constant $\alpha.$ (See, for example, \cite[Sec. 8.3, 8.4 after (8.4.51)]{BH}.)
\end{enumerate}
\end{remarks}

\begin{proof}
Fix $m\in P\cap\mathbb{Z}^n.$ By the lemma, for $s<0$
\[
\|\sigma^m_s\|^2_{L^2(K,L)} = (2\pi)^n \int_{K} e^{2|s|((x-m)\frac{\partial\psi}{\partial x}-\psi)}e^{2(m-x)y_0 - 2g_0}\,d^nx,
\]
which can be estimated for large $|s|$ by Laplace's method. In particular, for $s$ large and negative, the asymptotics are determined by the maximum of $(x-m)\frac{\partial\psi}{\partial x}-\psi$ on $K$. By assumption, suppose this unique maximum occurs at the point $x_0\in\partial K.$ Laplace's method then tells us that there exists $C_{K,m}>0$ such that for $s\rightarrow -\infty,$
\[
\|\sigma^m_s\|^2_{L^2(K,L)} \sim C_{K,m} |s|^{-\left(n+\tfrac12\right)}e^{2|s|\phi(x_0)}
\]
where $\phi(x_0)=(m-x_0)\frac{\partial\psi}{\partial x}(x_0)-\psi(x_0),$ and $C_{K,m}$ is an $s$-independent constant\footnote{In fact, $C_{K,m}=(2\pi)^{\frac{2n-1}{2}} \exp(2(m-x_0)y_0(x_0) - 2g_0(x_0)) / |J(x_0)|^{\frac12},$ where $J$ is a certain function depending on $\psi$, $g_0$ and a choice of parameterization of $\partial K$ near $x_0.$ See \cite[Sec. 8.3]{BH} for details.} \cite[eqn. 8.3.64]{BH}.

Consider now a slightly larger $K'\supset K$ still contained in $U^{(0,n)}_s$ (sufficiently close to $K$ such that our hypothesis hold also for $K'$). Since $\phi$ is increasing along the ray from $m$ through $x_0$, the unique maximum of $\phi$ on $\partial K'$, which occurs at some point $x'_0\in \partial K'$, is strictly larger than $\phi(x_0).$ Hence, we see that there exists constants $C:=C_{K,m}/C_{K',m}>0$ and $\alpha := \phi(x'_0) - \phi(x_0)$ such that for $s\rightarrow -\infty$,
\[
\frac{\|\sigma^m_s\|_{L^2(K,L)}^2}{\|\sigma^m_s\|_{L^2(K',L)}^2} \sim C e^{-|s|\alpha}
\]

Since $\|\sigma^m_s\|^2_{L^2(X_{P},L)} > \|\sigma^m_s\|^2_{L^2(K',L)}$, the above estimate holds if we replace $K'$ by $X$ (with a possibly larger $\alpha$ and different $C>0$). In particular, if we denote the $L^2$-normalized section $\sigma^m_s/\|\sigma^m_s\|_{L^2(X_P,L)}$ by $\hat{\sigma}^m_s$, then we have shown that there exist constants $C_{K,m},\alpha_{K,m}>0$ such that
\begin{equation}
\label{eqn:asymp2}
\|\hat{\sigma}^m_s\|^2_{L^2(K,L)} \sim C_{K,m}e^{-|s|\alpha_{K,m}}
\end{equation}
as $s\rightarrow -\infty.$

Finally, since the sections $\hat\sigma^m_s,\ m\in P\cap\mathbb{Z}^n$ form an orthonormal basis for $H^0_s(X_{P},L)$, there exists coefficients $d_m\in\mathbb{C},\ m\in P\cap\mathbb{Z}^n$ such that $u=\sum_m d_m\hat\sigma^m_s$ and $\|u\|^2_{L^2(X_{P},L)}=\sum_m |d_m|^2.$ Let $m_0$ be such that $\alpha:=\alpha_{K,m_0}=\min\{\alpha_{K,m}:m\in P\cap\mathbb{Z}^n\}$ and $m_1$ be such that $C := C_{K,m_1} = \max\{C_{m,K}:m\in P\cap\mathbb{Z}^n\}.$ The $L^2$-norm of $u$ restricted to $K$ is
\begin{align*}
\|u\|^2_{L^2(K,L)} &= \sum_{m,m'} \bar d_m d_{m'}\int_{K\times\mathbb{T}^n}\overline{\hat\sigma^m_s}\hat\sigma^{m'}_s d^n\theta\,d^nx \\
&= (2\pi)^n \sum_m |d_m|^2 \int_{K} |\hat\sigma^m_s|^2\,d^nx\\
&= \sum_m |d_m|^2 \|\hat{\sigma}^m_s\|^2_{L^2(K,L)}\\
&<  C_{K,m_0} e^{-|s|\alpha_{K,m_0}} \frac{C_{K,m_1}}{C_{K,m_0}} \|u\|^2_{L^2(X_{P},L)}
\end{align*}
where the last inequality follows from \ref{eqn:asymp2} and the fact that we take $s$ sufficiently negative.
\end{proof}

\begin{remark}
One has $H_{s}^{0}(X_{P},L):=\ker\left.  \bar{\partial}_{s}\right\vert_{\Gamma(L)}\subset L^{2}(X_{P},L)$, which means in particular that norm on the right-hand side of (\ref{eqn:1}) is finite.
\end{remark}

\noindent\textbf{Acknowledgments}

We would like to thank G. Marinescu for useful discussions and for showing us the proof of a semiclassical version of Theorem \ref{thm1}. The second two authors are supported by FCT/Portugal through the projects PEst-OE/EEI/LA009/2013, EXCL/MAT-GEO/0222/2012, \- PTDC/MAT/\-119689/2010, PTDC/MAT/1177762/2010. The second author is thankful for generous support from the Emerging Field Project on Quantum Geometry from Erlangen--N\"urnberg University.

\providecommand{\bysame}{\leavevmode\hbox to3em{\hrulefill}\thinspace}

\end{document}